\documentclass[12pt]{amsart}

\usepackage[latin1]{inputenc}
\usepackage{amssymb,amsmath,amsfonts, amsthm}
\usepackage{enumerate}
\usepackage{times}

\usepackage{bbding}
\usepackage[margin=2.0cm]{geometry}
\usepackage{cite}

\newtheorem{theorem}{Theorem}
\newtheorem{lemma}[theorem]{Lemma}
\newtheorem*{claim}{Claim}
\newtheorem{proposition}[theorem]{Proposition}
\newtheorem{corollary}[theorem]{Corollary}

\def\abs#1{\vert#1\vert}
\newcommand{\infconv}{\mathbin{\Box}} 

\newcommand{\cP}{{\mathcal P}} 
 \newcommand{\cK}{{\mathcal K}} 
 \newcommand{\cCe}{{\mathcal C}_{\text{epi}}^{n+1}} 
\newcommand{\fconvfpa}{\operatorname{Conv_{p.a.}}(\R^n;\R)} 
\newcommand{\fconvpa}{\operatorname{Conv_{p.a.}}(\R^n)} 

\newcommand{\MD}{\operatorname{det}} 
\newcommand{\Hess}{\operatorname{D}^2\!} 
\newcommand{\sln}{\operatorname{SL}(n)}

\DeclareMathOperator{\oZ}{\operatorname{Z}}
\DeclareMathOperator{\otZ}{\operatorname{\tilde Z}}
\DeclareMathOperator{\obZ}{\operatorname{\bar Z}}

\newcommand{\sq}{\mathbin{\vcenter{\hbox{\rule{.3ex}{.3ex}}}}} 

\newcommand{\R}{{\mathbb R}}
\newcommand{\N}{{\mathbb N}}

\newcommand{\dom}{\operatorname{dom}}

\newcommand{\hm}{{\mathcal H}}

\renewcommand{\d}{\,\mathrm{d}}
\newcommand{\ind}{{\rm\bf I}}

\newcommand{\fconv}{{\mbox{\rm Conv}_{\text{\rm coe}}(\R^n)}} 
\newcommand{\fconvs}{{\mbox{\rm Conv}_{{\rm sc}}(\R^n)}} 
\newcommand{\fconvsone}{{\mbox{\rm Conv}_{{\rm sc}}(\R)}} 
\newcommand{\fconvfone}{{\mbox{\rm Conv}(\R; \R)}} 
\newcommand{\fconvx}{{\mbox{\rm Conv}(\R^n)}}
\newcommand{\fconvi}{{\mbox{\rm Conv}_{\text{\rm od}}(\R^n)}}
\newcommand{\fconvf}{{\mbox{\rm Conv}(\R^n; \R)}}

\begin{document}

\title[]{A homogeneous decomposition theorem\\ for valuations on convex functions}

\author{Andrea Colesanti}
\address{Dipartimento di Matematica e Informatica ``U. Dini''
Universit\`a degli Studi di Firenze,
Viale Morgagni 67/A - 50134, Firenze, Italy}
\email{andrea.colesanti@unifi.it}

\author{Monika Ludwig}
\address{Institut f\"ur Diskrete Mathematik und Geometrie,
Technische Universit\"at Wien,
Wiedner Hauptstra\ss e 8-10/1046,
1040 Wien, Austria}
\email{monika.ludwig@tuwien.ac.at}

\author{Fabian Mussnig}
\address{Institut f\"ur Diskrete Mathematik und Geometrie,
Technische Universit\"at Wien,
Wiedner Hauptstra\ss e 8-10/1046,
1040 Wien, Austria}
\email{fabian.mussnig@alumni.tuwien.ac.at}


\date{}

\begin{abstract} 
The existence of a homogeneous decomposition for continuous and epi-translation invariant valuations on super-coercive functions is established. Continuous and epi-translation invariant valuations that are epi-homogeneous of degree $n$ are classified. By duality, corresponding results are obtained for valuations on finite-valued convex functions.

\end{abstract}


\maketitle

\section{Introduction}

Given a space of real-valued functions $X$, we consider real-valued valuations on $X$, that is, functionals $\oZ\colon X\to\R$ such that
\begin{equation}\label{val_def}
\oZ(u\vee v)+\oZ(u\wedge v)=\oZ(u)+\oZ(v)
\end{equation}
for every $u,v\in X$ with $u\vee v$ and $u\wedge v\in X$, where $\vee$ and $\wedge$ denote the point-wise maximum and 
minimum, respectively. For $X$, the space of indicator functions of convex bodies (that is, compact convex sets) in $\R^n$, we obtain the classical notion of valuation on convex bodies. 
Here strong structure and classification theorems have been established over the last seventy years (see  \cite{Alesker99,Alesker01,Bernig:Fu, Haberl_sln,  Ludwig:Reitzner2,  HaberlParapatits_moments, Haberl:Parapatits_centro, BernigFuSolanes} for some recent results and  \cite{Hadwiger:V,Klain:Rota, Schneider:CB2}  for information on the classical theory). The aim of this article is to obtain such results also in the functional setting. In particular, we will establish a homogeneous decomposition result \`a la McMullen \cite{McMullen77}.

Valuations on function spaces have only recently started to attract attention. Classification results were obtained for $L_p$ and Sobolev spaces \cite{LiMa, Tsang:Lp, Ludwig:SobVal, Ludwig:Fisher, Kone, Ma2016, Tuo_Wang_semi},  spaces of quasi-convex functions \cite{ColesantiLombardi, ColesantiLombardiParapatits}, of Lipschitz functions \cite{ColesantiPagniniTradaceteVillanueva}, of definable functions \cite{BaryshnikovGhristWright} and on Banach lattices \cite{TradaceteVillanueva_imrn}. Spaces of convex functions play a special role because of their close connection to convex bodies. Here classification results were obtained for $\sln$ invariant and for monotone valuations in \cite{ColesantiLudwigMussnig, ColesantiLudwigMussnig17, Mussnig19, Mussnig_super, Mussnig_log, CavallinaColesanti} and the connection to valuations on convex bodies was explored by Alesker \cite{Alesker_cf}. While the theory of translation invariant valuations is well developed for convex bodies, for convex functions the corresponding theory did not exist till now. We introduce the notion of {\em epi-translation invariance} to build such a theory. In particular, we will show that on the space of super-coercive convex functions there is a homogeneous decomposition for continuous and  epi-translation invariant valuations and there exist non-trivial such valuations for each degree of epi-homogeneity while on the larger space of coercive convex functions all continuous and epi-translation invariant valuations are constant. 

The general space of (extended real-valued) convex functions on $\R^n$ is defined as
\begin{equation*}
\fconvx=\{u\colon\R^n\to\R\cup\{+\infty\}\colon u\ \mbox{ is convex and lower semicontinuous},\ u\not\equiv+\infty\}.
\end{equation*}
It is equipped with the topology induced by  epi-convergence (see Section \ref{prelim_functions}). Continuity of valuations defined on $\fconvx$, or on subsets
of $\fconvx$, will be always with respect to this topology. The space $\fconvx$ is a standard space in convex analysis (see \cite{RockafellarWets}) and important in many applications. 
As we will show, $\fconvx$ is too large for our purposes. We will be mainly interested in two of its subspaces. The first is formed by {\em coercive} functions,
\begin{equation*}
\fconv=\left\{u\in\fconvx\colon\lim_{\vert  x\vert \to+\infty}u(x)=+\infty\right\},
\end{equation*}
where $\vert x\vert$ is the Euclidean norm of $x\in\R^n$.
The second is formed by {\em super-coercive} functions,
\begin{equation*}
\fconvs=\left\{u\in\fconvx\colon\lim_{\vert  x\vert \to+\infty}\frac{u(x)}{\vert  x\vert }=+\infty\right\}.
\end{equation*}
The space of super-coercive convex functions is related to another subspace of $\fconvx$, formed by convex functions with finite values,
$$
\fconvf=\big\{v\in\fconvx\colon v(x)<+\infty\,\text{ for all }\,x\in\R^n\big\}.
$$
Indeed, $v\in\fconvf$ if and only if its standard conjugate or Legendre transform $v^*$ belongs to $\fconvs$ (see Section \ref{dual}). 
\goodbreak

\subsection{}
One of the most important structural results for valuations on convex bodies is the existence of a homo\-geneous decomposition for translation invariant valuations. It was conjectured  by Hadwiger and established
by McMullen \cite{McMullen77} (see Section \ref{prelim_bodies}). Our first aim is to establish such a result for valuations on convex functions. We define \emph{epi-multiplication} by setting for $u\in \fconvx$ and $\lambda>0$, 
$$
\lambda \sq u(x)=\lambda\, u\left(
\frac x\lambda
\right)
$$ 
for $x\in\R^n$. From a geometric point of view, this operation has the following meaning: the epigraph of $\lambda \sq u$ is obtained by  rescaling the epigraph of $u$ by the factor $\lambda$. We extend the definition of epi-multiplication to $0\sq u(x) = 0$ if $x=0$ and $0\sq u(x)=+\infty$ if $x\neq 0$. It is easy to see that $u\in\fconvs$ implies $\lambda \sq u\in\fconvs$ for $\lambda\geq 0$.
A functional $\oZ:\fconvs\to \R$ is called \emph{epi-homogeneous} of degree $\alpha\in\R$ if
$$
\oZ(\lambda \sq u)=\lambda^\alpha \oZ(u)
$$
for all $u\in\fconvs$ and $\lambda> 0$. Here and in the following corresponding definitions will be used for $\fconvx$ and its subspaces.

We call   $\oZ:\fconvs\to \R$ \emph{translation invariant} if
$\oZ(u\circ\tau^{-1})=\oZ(u)$ 
for every $u\in\fconvs$ and every translation $\tau:\R^n\to\R^n$. If $u\in\fconvs$ then $u\circ \tau^{-1}\in\fconvs$ as well.  
We say that $\oZ$ is {\em vertically translation invariant} if 
$$
\oZ(u+\alpha)=\oZ(u)
$$
for  all $u\in\fconvs$ and $\alpha\in\R$. If $\oZ$ is both translation invariant and vertically translation invariant, then $\oZ$ is called \emph{epi-translation invariant}. As we will see, the set of continuous, epi-translation invariant valuations on $\fconvs$ is non-empty. Note that a functional $\oZ$ is epi-translation invariant if for all $u\in\fconvs$ the value $\oZ(u)$ is not changed by translations of the epigraph of $u$.

The following result establishes a homogeneous decomposition for continuous and  epi-translation invariant  valuations on $\fconvs$.

\begin{theorem}\label{McM sc} If $\,\oZ: \fconvs\to \R$ is a continuous and epi-translation invariant valuation, then there are continuous and epi-translation invariant valuations $\oZ_0, \dots, \oZ_n: \fconvs \to \R$ such that 
$\oZ_i$ is epi-homogeneous of degree  $i$  and 
$\oZ=\oZ_0+\dots +\oZ_n$.
\end{theorem}

We will see that this theorem is no longer true if we remove the condition of vertical translation invariance (see Section \ref{vert}). We will also see that the set of continuous and  epi-translation invariant valuations is trivial on the larger set of coercive convex functions (see Section \ref{coercive}). Hence the assumption of super-coercivity is in some sense necessary. 

\goodbreak
Milman and Rotem \cite{Milman_Rotem_Mixed_Integrals} discuss the problem to find a functional analog of Minkowski's  mixed volume theorem. In particular, they point out that such a result is not possible on $\fconvx$ for inf-convolution as addition and the volume functional $u\mapsto \int_{\R^n} e^{-u(x)} \d x$.  Instead,  they define a new addition for convex functions to obtain a functional mixed volume theorem. A consequence of Theorem \ref{McM sc} is that continuous and epi-translation invariant valuations are multi\-linear on $\fconvs$ with respect to inf-convolution and epi-multiplication  (see Theorem~\ref{thm poly sc}). Thus, for all such valuations, a functional analog of Minkowski's mixed volume  theorem  is obtained on $\fconvs$ with inf-convolution as addition.

\goodbreak
\subsection{}
The following result gives a characterization of continuous and epi-translation invariant valuations on $\fconvs$, which are epi-homogeneous of degree $n$. For $u\in\fconvs$, we denote by $\dom(u)$ the set of points of $\R^n$ where $u$ is
finite and by $\nabla u$ the gradient of $u$. Note that by standard properties of convex functions, $\nabla u(x)$ is well defined for a.e.\! $x\in\dom(u)$.  Let $C_c(\R^n)$ be the set of continuous functions with compact support on $\R^n$.

\begin{theorem}\label{theorem n-homogeneous} 
A functional $\oZ\colon\fconvs\to\R$ is a continuous and  epi-translation invariant valuation  that is epi-homogeneous of degree $n$, if and only if there exists $\zeta\in C_c(\R^n)$ such that
$$
\oZ(u)=\int_{\dom(u)}\zeta(\nabla u(x))\d x
$$
for every $u\in\fconvs$.
\end{theorem}
\goodbreak

We will also obtain a classification of continuous and epi-translation invariant valuations that are epi-homogeneous of degree $0$. These are just constants. As a consequence of these results and  Theorem \ref{McM sc}, we obtain the following complete classification in dimension one.

\begin{corollary} A functional $\,\oZ\colon\fconvsone\to\R$ is a continuous and epi-translation invariant valuation, if and only if there
exist a constant $ \zeta_0\in\R$ and a function $\zeta_1\in C_c(\R)$ such that
$$
\oZ(u)=\zeta_0+\int_{\dom(u)}\zeta_1(u'(x))\d x$$
for every $u\in\fconvsone$.
\end{corollary}

\subsection{}\label{dual}
As mentioned before, there exists a bijection between $\fconvf$ and $\fconvs$ given by the standard conjugate, or Legendre transform, of convex functions. For $u\in\fconvx$, we denote by $u^*$ its conjugate, defined by
$$
u^*(y)=\sup\nolimits_{x\in\R^n}\left(\langle x,y\rangle-u(x)\right)$$
for $y\in\R^n$, where $\langle x,y\rangle$ is the inner product of $x, y\in\R^n$. Note that $u\in\fconvs$ if and only if $u^*\in\fconvf$ 
(see, for example, \cite[Theorem 11.8]{RockafellarWets}).

Let $\oZ$ be a continuous valuation on $\fconvf$. It was proved in \cite{ColesantiLudwigMussnig3} that $\oZ^*\colon\fconvs\to\R$, defined by
$$
\oZ^*(u)=\oZ(u^*),
$$
 is a continuous valuation as well. This fact permits to transfer results for valuations on $\fconvf$ to results valid for valuations on 
$\fconvs$ and vice versa.  We call $\oZ^*$ the dual valuation of $\oZ$.

A valuation $\oZ$ on $\fconvf$ is called \emph{homogeneous}  if there exists $\alpha\in\R$ such that
$$
\oZ(\lambda v)=\lambda^\alpha\oZ(v)
$$
for all $v\in\fconvf$ and $\lambda\geq 0$. We say that $\oZ$ is \emph{dually translation invariant} if for every linear function $\ell\colon\R^n\to\R$ 
$$
\oZ(v+\ell)=\oZ(v)
$$
for every $v\in\fconvf$. Let $\ell(y)=\langle y, x_0\rangle$ for $x_0,y\in\R^n$. As $(v+\ell)^*(x) = v^*(x-x_0)$ for $v\in \fconvf$, we see that $\oZ$ is dually translation invariant if and only if $\oZ^*$ is translation invariant.
We define vertical translation invariance for valuations on $\fconvf$ in the same way as on $\fconvs$. We say that $\oZ$ is \emph{dually epi-translation invariant} on $\fconvf$ if it is vertically and dually translation invariant. Note that a functional $\oZ$ is dually epi-translation invariant, if for all $v\in\fconvf$, the value $\oZ(v)$ is not changed by adding an affine function to $v$.

\goodbreak
Let $\oZ$ be a valuation on $\fconvf$. We note the following simple facts. The valuation $\oZ$ is vertically translation invariant if and only if $\oZ^*$ has the same property. The valuation $\oZ^*$ is epi-homogeneous of degree $\alpha$ if and only if $\oZ$ is homogeneous of degree $\alpha$. 

\goodbreak
Hence we obtain the following result as a consequence of Theorem \ref{McM sc}.

\begin{theorem}\label{McM finite} If $\,\oZ: \fconvf\to \R$ is a continuous and dually epi-translation invariant valuation, then there are continuous and  dually epi-translation invariant valuations $\oZ_0, \dots, \oZ_n: \fconvf \to \R$ such that 
$\oZ_i$ is homogeneous of degree  $i$  and 
$\oZ=\oZ_0+\dots +\oZ_n$.
\end{theorem}
\medskip

Alesker \cite{Alesker_cf} introduced the following class of valuations on $\fconvf$. 
Given  real symmetric $n\times n$ matrices $M_1,\dots,M_n$, denote by $\MD(M_1,\dots,M_n)$ their mixed discriminant. Let $i\in\{1,\dots,n\}$ and write $\MD(M[i],M_1,\dots,M_{n-i})$ for  the mixed discriminant in which the matrix $M$ is repeated $i$ times.
Let $A_1,\dots,A_{n-i}$ be continuous, symmetric $n\times n$ matrix-valued functions on $\R^n$ with compact support and $\zeta\in C_c(\R^n)$. Given a function $v\in\fconvf\cap C^2(\R^n)$, set
\begin{equation}\label{val A}
\oZ(v)=\int_{\R^n} \zeta(x) \det(\Hess v(x)[i], A_1(x),\dots,A_{n-i}(x))\d x
\end{equation}
where $\Hess v$ is the Hessian matrix of $v$. Alesker \cite{Alesker_cf} proved that $\oZ$ can be extended to a continuous valuation on $\fconvf$. 
Valuations of type \eqref{val A} are  homogeneous of degree $i$ and dually epi-translation invariant. This implies in particular that the set of valuations with these properties is non-empty.  Clearly,
the dual functional $\oZ^*$ is a continuous, epi-translation invariant, epi-homogeneous valuation on $\fconvs$.

Next, we  state the counterpart of Theorem \ref{theorem n-homogeneous} for valuations on $\fconvf$. Let
$\Theta_0(v,\cdot)$ be the Hessian measure of order $0$ of a function $v\in\fconvf$ (see Section \ref{section hessian measures} for the definition).

\begin{theorem}\label{theorem n-homogeneous finite} A functional $\oZ\colon\fconvf\to\R$ is a continuous and dually epi-translation invariant valuation that is homogeneous of degree $n$, if and only if there exists $\zeta\in C_c(\R^n)$ such that
$$
\oZ(v)=\int_{\R^n\times \R^n}\zeta(x)\d\Theta_0(v,(x,y))$$
for every $v\in\fconvf$. 
\end{theorem}

In the special case of dimension one, we obtain the following complete classification theorem.

\begin{corollary} A functional $\,\oZ\colon\fconvfone\to\R$ is a continuous and dually epi-translation invariant valuation, if and only if there
exist a constant $ \zeta_0\in\R$ and a function $\zeta_1\in C_c(\R)$ such that
$$
\oZ(v)=\zeta_0+\int_{\R\times \R}\zeta_1(x)\d\Theta_0(v,(x,y))$$
for every $v\in\fconvfone$.
\end{corollary}

The plan for this paper is as follows. In Section \ref{prelim}, we collect results on convex bodies and functions needed for the proofs of the main results.  In Section \ref{inclu_exclu}, an inclusion-exclusion principle is established for valuations on convex functions and  in Section \ref{section hessian measures}, the existence and properties of the valuations in Theorem \ref{theorem n-homogeneous} and Theorem \ref{theorem n-homogeneous finite} are deduced by using results on Hessian valuations.   Theorem \ref{McM sc} is proved  in Section \ref{proof1}. As a consequence the polynomiality of epi-translation invariant valuations is obtained and  a connection to the valuations introduced by Alesker is established in Section \ref{section polynomiality}.  The proof of Theorem \ref{theorem n-homogeneous} is given in Section \ref{class}.  In the final sections, the necessity of the assumptions in Theorem \ref{McM sc} is demonstrated.

\goodbreak
\section{Preliminaries}\label{prelim}

We work in $n$-dimensional Euclidean space $\R^n$, with $n\ge 1$, endowed with the Euclidean norm $\vert \cdot\vert $ and the usual scalar product
$\langle \cdot,\cdot\rangle$. 

\subsection{}\label{prelim_bodies}
A {\em convex body} is a nonempty, compact and convex subset of $\R^n$. The family of all convex bodies is denoted by $\cK^n$. A
{\em polytope} is the convex hull of finitely many points in $\R^n$. The set of polytopes, denoted by $\cP^n$, is  contained in $\cK^n$. We equip both $\cK^n$ and $\cP^n$ with the topology coming from the Hausdorff metric.

A functional $\oZ: \cK^n \to \R$ is a valuation if  
\begin{equation*}
\oZ(K\cup L)+\oZ(K\cap L)=\oZ(K)+\oZ(L)
\end{equation*}
for every $K,L\in \cK^n$ with  $K\cup L \in \cK^n$. We say that $\oZ$ is translation invariant if
$\oZ(\tau K) = \oZ(K)$
for all translations $\tau: \R^n\to \R^n$ and $K\in\cK^n$. It is homogeneous of degree $\alpha\in \R$, if $\oZ(\lambda\,K)= \lambda^\alpha\oZ(K)$ for all $K\in\cK^n$ and $\lambda\geq 0$.

\goodbreak
The following result by McMullen \cite{McMullen77} establishes a homogeneous decomposition for continuous and  translation invariant  valuations on $\cK^n$. 

\begin{theorem}[McMullen]\label{McM} If $\,\oZ: \cK^n\to \R$ is a continuous and translation invariant valuation, then there are continuous and translation invariant valuations $\oZ_0, \dots, \oZ_n: \cK^n \to \R$ such that 
$\oZ_i$ is homogeneous of degree  $i$  and 
$\oZ=\oZ_0+\dots +\oZ_n$.
\end{theorem}

We recall two classification results for valuations on convex bodies. First, we note that it is easy to see that every 
 continuous  and translation invariant valuation  that is homogeneous of degree $0$ is constant. 
 The classification of continuous and translation invariant valuations that are  $n$-homogeneous  is due to Hadwiger  \cite{Hadwiger:V}. Let $V_n$ denote $n$-dimensional volume (that is, $n$-dimensional Lebesgue measure).

\begin{theorem}[Hadwiger]\label{n-homogeneous} 
A functional $\oZ\colon\cK^n\to\R$ is a continuous and translation invariant valuation  that is homogeneous of degree $n$, if and only if there exists $\alpha\in\R$ such that $\oZ=\alpha \, V_n$.
\end{theorem}

\subsection{}\label{prelim_functions} 
Given a subset $A\subset\R^n$, let $\ind_A\colon\R^n\to\R\cup\{+\infty\}$ denote the (convex) indicatrix function of $A$,
$$
\ind_A(x)=
\left\{
\begin{array}{lll}
\phantom{\,+}0\quad&\mbox{if $x\in A$,}\\
+\infty\quad&\mbox{if $x\notin A$.}
\end{array}
\right.
$$
Note that for a convex body $K$, we have $\ind_K\in\fconvs$.

We equip $\fconvx$ with the topology associated to epi-convergence.
Here a sequence $u_k\in\fconvx$ is \emph{epi-convergent}  to $u\in\fconvx$ if for all $x\in\R^n$ the following conditions hold:
\begin{itemize}
	\item[(i)] For every sequence $x_k$ that converges to $x$, we have $u(x) \leq \liminf_{k\to \infty} u_k(x_k)$.
			
	\item[(ii)] There exists a sequence $x_k$ that converges to $x$ such that $u(x) = \lim_{k\to\infty} u_k(x_k)$.
\end{itemize}
The following result can be found in \cite[Theorem 11.34]{RockafellarWets}.

\begin{proposition}\label{convergence conjugates} A sequence $u_k$ of functions from $\fconvx$ epi-converges to $u\in\fconvx$
if and only if the sequence $u_k^*$ epi-converges to $u^*$. 
\end{proposition}

If $u\in\fconv$, then for $t\in \R$ the sublevel sets $\{u\leq t\}=\{x\in\R^n\colon u(x)\leq t\}$ are either empty or in $\cK^n$. The next result, which follows from \cite[Lemma 5]{ColesantiLudwigMussnig} and \cite[Theorem 3.1]{BeerRockafellarWets}, shows that on $\fconv$ epi-convergence is equivalent to Hausdorff convergence of sublevel sets, where we say that $\{u_k\leq t\}\to \emptyset$ as $k\to\infty$ if there exists $k_0\in\N$ such that $\{u_k\leq t\}=\emptyset$ for $k\geq k_0$.

\begin{lemma}\label{hd conv lvl sets}
If $u_k,u\in\fconv$, then $u_k$ epi-converges to $u$ if and only if $\{u_k \leq t\} \to \{u\leq t\}$ for every $t\in\R$ with $t\neq \min_{x\in\R^n} u(x)$.
\end{lemma}

\subsection{}
A function $v\in\fconvf$ is called {\em piecewise affine} if there exist finitely many affine functions $w_1,\dots,w_m: \R^n \to \R$ such that
\begin{equation}\label{pa function1}
v=\bigvee_{i=1}^m w_i.
\end{equation}
The set of piecewise affine functions will be denoted by $\fconvfpa$. It is a subset of $\fconvf$.    

\goodbreak
We recall that epi-convergence in $\fconvf$ is equivalent to uniform convergence on compact sets (see, for example, \cite[Theorem 7.17]{RockafellarWets}). Hence the following proposition
follows from standard approximation results for convex functions. 

\begin{proposition}
\label{approximation lemma finite}
For every $v\in\fconvf$, there exists a sequence in $\fconvfpa$ which epi-converges to $v$.
\end{proposition}

We also need to introduce a counterpart of $\fconvfpa$ in  $\fconvs$.
For given polytopes $P,P_1,\dots,P_m\in\cP^n$, the collection $\{P_1,\dots,P_m\}$
is called a {\em polytopal partition} of $P$ if
$
P=\bigcup_{i=1}^m P_i$ and the $P_i$'s have pairwise disjoint interiors. A function $u\in\fconvs$ belongs to $\fconvpa$ if there exists a polytope $P$ and a polytopal partition 
$\{P_1,\dots,P_m\}$ of $P$ such that
\begin{equation*}
u=\bigwedge_{i=1}^m(w_i+\ind_{P_i})
\end{equation*}
where $w_1, \dots, w_m: \R^n \to \R$ are affine.
\goodbreak

By \cite[Theorem 11.14]{RockafellarWets}, a function $u$ is in $\fconvpa$ if and only if $u^*$ is in $\fconvfpa$. Hence, we obtain the following consequence of Proposition \ref{convergence conjugates} and Proposition \ref{approximation lemma finite}.

\begin{corollary}\label{approximation lemma}
For every $u\in\fconvs$, there exists a sequence in $\fconvpa$ which epi-converges to $u$.
\end{corollary}
Since $\fconvs$ is a dense subset of $\fconv$, it is easy to see that the statement of Corollary~\ref{approximation lemma} also holds if $\fconvs$ is replaced by $\fconv$.

\section{The Inclusion-Exclusion Principle}\label{inclu_exclu}

It is often useful to extend the valuation property (\ref{val_def}) to several convex functions. For valuations on convex bodies, this is an important tool and a consequence of Groemer's extension theorem \cite{Groemer78}.
For $m\geq 1$ and $u_1, \dots, u_m\in\fconvx$, we set $u_J=\bigvee_{j\in J}u_j$ for $\emptyset\neq J\subset\{1,\ldots,m\}$. Let $\abs{J}$ denote the number of elements in $J$.

\begin{theorem}\label{inclusion_exclusion}
  If\, $\oZ:\fconvx \to \R$ is a continuous valuation, then
    \begin{equation}\label{in_ex}
    \oZ(u_1\wedge \dots\wedge u_m)=\sum_{\emptyset\neq
      J\subset\{1,\ldots,m\}}(-1)^{\abs{J} -1} \oZ(u_J)
  \end{equation}
for all $u_1, \dots, u_m\in\fconvx$ and $m\in\N$  whenever $u_1\wedge \dots\wedge u_m\in \fconvx$.
\end{theorem}

Note that $\fconv$ and $\fconvs$ are closed under the operation of taking maxima. Hence Theorem \ref{inclusion_exclusion} also holds with $\fconvx$ replaced by one of these spaces.

\goodbreak
Let $\bigwedge\fconvx$ denote the set of finite minima of convex functions from $\fconvx$. It is easy to see that $\bigwedge \fconvx$ is a lattice. If $\oZ$ is a valuation on a lattice, a simple induction argument shows that the inclusion-exclusion principle (\ref{in_ex}) holds. Hence Theorem \ref{inclusion_exclusion} is a consequence of the following extension result.

\begin{theorem}\label{extend_function} A continuous valuation on $\fconvx$ admits a unique extension to a
valuation on the lattice $\bigwedge\fconvx$.
\end{theorem}

We identify a convex function with its epigraph. Let $\cCe$ be the set of closed convex sets in $\R^{n+1}$ that are epigraphs of functions in $\fconvx$ and  equip this set  with the Painlev\'e-Kuratowski topology, which corresponds to the topology induced by epi-convergence (see, for example, \cite[Definition 7.1]{RockafellarWets}). A slight modification of Groemer's extension theorem \cite{Groemer78} (or see \cite[Theorem 6.2.3]{Schneider:CB2}  or \cite{Klain:Rota}) shows that the following statement is true (we omit the proof). Here $\bigcup\cCe$ is the set of all finite unions of elements from $\cCe$. Theorem~\ref{extend_function} is equivalent to Theorem \ref{extend}.

\begin{theorem}\label{extend} A continuous valuation on $\cCe$ admits a unique extension to a
valuation on the lattice $\bigcup\cCe$.
\end{theorem}

We require the following simple consequence of the inclusion-exclusion principle, Theorem \ref{inclusion_exclusion} and of Corollary \ref{approximation lemma}.

\begin{lemma}\label{null lemma}
Let $\oZ$ be a continuous valuation on $\fconvs$ (or on $\fconv$). If
\begin{equation}\label{null assumption}
\oZ(w+\ind_P)=0
\end{equation}
for every affine function $w: \R^n\to \R$ and for every polytope $P$, then $\oZ\equiv0$. 
\end{lemma}

\begin{proof} By Corollary \ref{approximation lemma} (and the remark following it), it suffices to prove that $\oZ(u)=0$ for\linebreak$u\in\fconvs$ (or $u\in\fconv$) that is piecewise affine. So, let $u= \bigwedge_{i=1}^m (w_i+\ind_{P_i})$ with $w_1, \dots, w_m$ affine and $P_1, \dots, P_m\in\cP^n$. 
By  Theorem \ref{inclusion_exclusion} (and the remark following it), it is enough to show that
$$\oZ\left(\bigvee_{j\in J} (w_j+\ind_{P_j})\right)=0$$
for every $\emptyset\ne J\subset \{1,\dots, m\}$. This follows from (\ref{null assumption}) as  $\bigvee_{j\in J} (w_j+\ind_{P_j})$ is a piecewise affine function restricted to a polytope.
\end{proof}

\section{Hessian measures and valuations}\label{section hessian measures}

For $u\in\fconvx$ and $x\in\R^n$, we denote by $\partial u(x)$ the subgradient of $u$ at $x$, that is,
$$
\partial u(x)=\{y\in\R^n\colon u(z)\ge u(x)+\langle z-x,y\rangle\, \text{ for all }\,z\in\R^n\}.
$$
We set
$$
\Gamma_u=\{(x,y)\in\R^n\times\R^n\colon y\in\partial u(x)\}.
$$
In other words, $\Gamma_u$ is the generalized graph of $\partial u$.  

Next, we recall the notion of Hessian measures of a function $u\in\fconvx$. These are non-negative Borel measures defined 
on the Borel subsets of $\R^n\times\R^n$, which we will denote by $\Theta_i(u,\cdot)$ with $i=0,\dots,n$. Their definition can be given as 
follows (see also \cite{ColesantiHug2000, ColesantiHug2000a, ColesantiLudwigMussnig3}). Let $\eta\subset\R^n\times\R^n$ be a Borel set and $s>0$. Consider the following set
$$
P_s(u,\eta)=\{x+sy\colon(x,y)\in\Gamma_u\cap\eta\}.
$$
It can be proven (see Theorem 7.1 in \cite{ColesantiLudwigMussnig3}) that $P_s(u,\eta)$ is measurable and that its measure is a polynomial
in the variable $s$, that is, there exists $(n+1)$ non-negative coefficients $\Theta_i(u,\eta)$ such that
$$
\hm^n(P_s(u,\eta))=\sum_{i=0}^n\binom ni s^i\Theta_{n-i}(u,\eta).
$$ 
Here $\hm^n$ is the $n$-dimensional Hausdorff measure in $\R^n$, normalized so that it coincides with the Lebesgue measure in $\R^n$.
The previous formula defines the Hessian measures of $u$; for more details we refer the reader to 
\cite{ColesantiHug2000, ColesantiHug2000a, ColesantiLudwigMussnig3}.

According to Theorem 8.2 in \cite{ColesantiLudwigMussnig3},  for every $v\in\fconvf$ and for every Borel subset
$\eta$ of $\R^n\times\R^n$
\begin{equation}\label{Remark on Hessian measures}
\Theta_i(v,\eta)=\Theta_{n-i}(v^*,\hat \eta),
\end{equation}
where $
\hat\eta=\{(x,y)\in\R^n\times\R^n\colon (y,x)\in\eta\}$.

We require the following statement for Hessian valuations for $i=0$. As the proof is the same for all indices $i$, we give the more general statement. Let $[\Hess v(x)]_i$ be the $i$-th elementary symmetric function of the eigenvalues of the Hessian matrix $\Hess v$.

\begin{theorem}\label{theorem 1}
Let $\zeta\in C(\R\times\R^n\times\R^n)$ have compact support with respect to the second variable. For $i\in\{0,1,\dots,n\}$, 
\begin{equation}\label{eq:z_thm_1}
\oZ(v)=\int_{\R^n\times\R^n} \zeta(v(x),x,y)\d\Theta_i(v,(x,y))
\end{equation}
is well defined for every $v\in\fconvf$ and defines a continuous valuation on $\fconvf$. 
Moreover, 
\begin{equation}\label{smooth}
\oZ(v)=\int_{\R^n}\zeta(v(x),x,\nabla v(x))\ [\Hess v(x)]_{n-i}\d x
\end{equation}
for every $v\in\fconvf\cap C^2(\R^n)$.
\end{theorem}

\goodbreak

We use the following result.

\begin{theorem}[\!\!\cite{ColesantiLudwigMussnig3}, Theorem 1.1]\label{thm:hessian_measures_main_result}
Let $\zeta\in C(\R\times\R^n\times\R^n)$ have compact support with respect to the second and third variables. For every $i\in\{0,1,\ldots,n\}$, the functional defined by
$$v\mapsto \int_{\R^{n}\times\R^n} \zeta(v(x),x,y)\d\Theta_i(v,(x,y))$$
defines a continuous valuation on $\fconvx$. Moreover, 
\begin{equation}\label{smooth_x}
\oZ(v)=\int_{\R^n}\zeta(v(x),x,\nabla v(x))\ [\Hess v(x)]_{n-i}\d x
\end{equation}
for $v\in\fconvx\cap C^2(\R^n)$.
\end{theorem}

\begin{proof}[Proof of Theorem~\ref{theorem 1}]
Since $\zeta$ has compact support with respect to the second variable, there is $r>0$ such 
that $\zeta(t,x,y)=0$ for every $y\in\R^n$ with $|y|\ge r$ and  $(t,x)\in\R\times\R^n$. 
Let  $v, v_k\in\fconvf$ be such that $v_k$ epi-converges to $v$. Since the functions are convex and finite this implies uniform convergence on compact sets, in particular on $B_r:=\{x\in\R^n\,:\, |x|\leq r\}$. Moreover, the sequence $v_k$ is uniformly bounded on $B_r$ and  uniformly Lipschitz. Hence, there exists $c>0$ such that
$$|v_k(x)|\leq c,\;|v(x)|\leq c,\; |y|\leq c$$
for all $k\in\N$,  $x\in B_r$ and  $y\in\partial v_k(x)\cup \partial v(x)$.

Next, let $\eta:\R^n\to\R$ be smooth with compact support such that $\eta(y)=1$ for all $y\in\R^n$ with $|y|\leq c$ and define $\tilde{\zeta}\in C(\R\times \R^n\times \R^n)$ by
$$\tilde{\zeta}(t,x,y)=\zeta(t,x,y)\,\eta(y).$$
The function $\tilde{\zeta}$ satisfies the conditions of Theorem~\ref{thm:hessian_measures_main_result} and $\zeta(v(x),x,y)=\tilde{\zeta}(v(x),x,y)$ for all $ x\in\R^n$, $y\in\partial v(x)$ and
$\zeta(v_k(x),x,y)=\tilde{\zeta}(v_k(x),x,y)$ for all $x\in\R^n$, $y\in\partial v_k(x)$ and $k\in\N$.
Hence, by Theorem~\ref{thm:hessian_measures_main_result},
\begin{multline*}
\int_{\R^{n}\times\R^n} \zeta(v_k(x),x,y) \d \Theta_i(v_k,(x,y))= \int_{\R^{n}\times \R^n} \tilde{\zeta}(v_k(x),x,y) \d\Theta_i(v_k,(x,y))\\
{\longrightarrow} \int_{\R^{n}\times \R^n} \tilde{\zeta}(v(x),x,y) \d\Theta_i(v,(x,y)) = \int_{\R^{n}\times \R^n} \zeta(v(x),x,y) \d\Theta_i(v,(x,y))
\end{multline*}
as $k\to\infty$. Since $v$ and $v_k$ were arbitrary this shows that \eqref{eq:z_thm_1} is well defined and continuous. Since such a function $\tilde{\zeta}$ can especially be found for any finite number of functions in $\fconvf$, this also proves the valuation property. Property (\ref{smooth}) follows from (\ref{smooth_x}).
\end{proof}

\goodbreak
As a simple consequence of Theorem  \ref{theorem 1} we obtain the following statement.

\begin{proposition}\label{if part - general case} For $\zeta\in C_c(\R^n)$, 
the functional
$\oZ\colon\fconvf\to\R$, defined by
\begin{equation}
\label{z if part}
\oZ(v)=\int_{\R^n\times \R^n}\zeta(x)\d\Theta_0(v,(x,y)),
\end{equation}
is a continuous, dually epi-translation invariant valuation which is is homo\-geneous of degree $n$. 
\end{proposition}
\begin{proof}
By Theorem~\ref{theorem 1} the map defined by \eqref{z if part} is a continuous valuation on on $\fconvf$. It remains to show dually epi-translation invariance. For $v\in\fconvf\cap C^2(\R^n)$ it follows from \eqref{smooth} that
$$\oZ(v)=\int_{\R^n} \zeta(x) \ \det (\Hess v(x)) \d x$$
which is clearly invariant under the addition of constants and linear terms. The statement now easily follows for general $v\in\fconvf$ by approximation.
\end{proof}

\goodbreak
By the considerations presented in Section \ref{dual}, (\ref{Remark on Hessian measures}) and Proposition \ref{if part - general case} lead to the following result.

\begin{proposition}\label{prop 1}
For $\zeta\in C_c(\R^n)$,  the functional
$\oZ\colon\fconvs\to\R$, defined by
$$
\oZ(u)=\int_{\R^n\times \R^n}\zeta(y)\d\Theta_n(u,(x,y)),
$$ 
is a continuous and epi-translation invariant valuation on $\fconvs$ which is epi-homogeneous of degree $n$. 
\end{proposition}

Note, that if $\oZ$ is as in Proposition~\ref{prop 1}, then
$$
\oZ(u)=\int_{\R^n\times \R^n}\zeta(y)\d\Theta_n(u,(x,y))=\int_{\dom(u)}\zeta(\nabla u(x))\d x
$$
for every $u\in\fconvs$. See also \cite[Section 10.4]{ColesantiLudwigMussnig3}.

\section{Proof of Theorem \ref{McM sc}}\label{proof1}
For $y\in\R^n$, define the linear function $\ell_{y}\colon\R^n\to\R$ as
$$
\ell_{y}(x)=\langle x,y\rangle.
$$
For $K\in\cK^n$, the function $\ell_{y}+\ind_K$ belongs to $\fconvs$. 

\begin{claim}
The functional $\otZ_y\colon\cK^n\to\R$, defined by
$$
\otZ_y(K)=\oZ(\ell_{y}+\ind_K),
$$ is a continuous and translation invariant valuation. 
\end{claim}
\begin{proof}
\medskip

\noindent{\em i) The valuation property.} Let $K,L\in\cK^n$ be such that $K\cup L\in\cK^n$. Note that 
$$
(\ell_{y}+\ind_K)\vee (\ell_{y}+\ind_L)=\ell_{y}+\ind_{K\cap L};\quad
(\ell_{y}+\ind_K)\wedge (\ell_{y}+\ind_L)=\ell_{y}+\ind_{K\cup L}.
$$
Hence the valuation property of $\oZ$ implies that $\otZ_y$ is a valuation.

\medskip

\noindent{\em ii) Translation invariance.} Let $x_0\in\R^n$. For every $x\in\R^n$ we have
\begin{eqnarray*}
\ell_{y}(x)+\ind_{K+x_0}(x)&=&\langle x,y\rangle+\ind_K(x-x_0)\\
&=&\langle x-x_0,y\rangle+\ind_K(x-x_0)+\langle x_0,y\rangle\\
&=&\ell_{y}(x-x_0)+\ind_K(x-x_0)+\langle x_0,y\rangle.
\end{eqnarray*}
In other words, the functions $\ell_{y}+\ind_{K+x_0}$ and $\ell_{y}+\ind_K$ differ only by a translation of the variable and by 
an additive constant. Using the epi-translation invariance of $\oZ$ we get
$$
\otZ_y(K+x_0)=\oZ(\ell_{y}+\ind_{K+x_0})=\oZ(\ell_{y}+\ind_{K})=\otZ_y(K).
$$

\medskip
\noindent{\em iii) Continuity.} 
By Lemma~\ref{hd conv lvl sets}, a sequence of convex bodies $K_i$ converges to $K$ if and only if $\ell_{y}+\ind_{K_i}$ epi-converges to $\ell_{y}+\ind_{K}$.
Hence the continuity of $\oZ$ implies that of $\otZ_y$. 
\end{proof}

Let $y\in\R^n$ be fixed. By the previous claim and Theorem~\ref{McM}, 
there exist continuous and translation invariant valuations $\otZ_{y,0}, \dots, \otZ_{y,n}$ on $\cK^n$ such that $\otZ_{y,j}$ is $j$-homogeneous and
\begin{equation*}
\otZ_y=\sum_{j=0}^n\otZ_{y,j}.
\end{equation*}
Let $K\in\cK^n$. For $\lambda\geq 0$, we have $
\lambda \sq (\ell_{y}+\ind_K)=\ell_{y}+\ind_{\lambda K}$.
Therefore we obtain, for every $\lambda\geq 0$,
\begin{equation*}
\oZ(\lambda \sq (\ell_{y}+\ind_K))=\sum_{j=0}^n\otZ_{y,j}(K)\lambda^j.
\end{equation*}

We consider the system of equations,
\begin{equation}\label{system}
\oZ(k \sq (\ell_{y}+\ind_K))=\sum_{j=0}^n\otZ_{y,j}(K)k^j,\quad k=0,1,\dots,n.
\end{equation}
Its associated matrix is a Vandermonde matrix and invertible. Hence there are coefficients $\alpha_{ij}$ for $i,j=0,\dots n$, such that
$$
\otZ_{y, i}(K)=\sum_{j=0}^n \alpha_{ij}\oZ(k\sq (\ell_{y}+\ind_K)),\quad\, i=0,\dots, n.
$$
Note that the coefficients $\alpha_{ij}$ are independent of $y$ and $K$.

For $i=0,\dots,n$, we define $\oZ_i\colon\fconvs\to\R$ as
$$
\oZ_i(u)=\sum_{j=0}^n \alpha_{ij}\oZ(j \sq u).
$$ 
In general, if $\oZ$ is a continuous, epi-translation invariant valuation on $\fconvs$ and $\lambda\geq 0$, then 
the functional $u\mapsto\oZ(\lambda \sq u)$ is a continuous, epi-translation valuation as well. Hence $\oZ_i$ is a continuous, epi-translation invariant valuation on $\fconvs$, for every
$i=0,\dots,n$. 

By \eqref{system} and the definition of $\oZ_i$, for every $y\in\R^n$ and $K\in\cK^n$ we may write
$$
\oZ_i(\ell_{y}+\ind_K)=\otZ_{y,i}(K).
$$
Therefore
$$
\oZ(\ell_{y}+\ind_K)=\sum_{i=0}^n\oZ_i(\ell_{y}+\ind_K).
$$
Moreover, by the homogeneity of the $\oZ_{y,i}$ we have, for $\lambda\geq 0$,
$$
\oZ_i(\lambda \sq (\ell_{y}+\ind_K))=
\otZ_{y, i}(\lambda K)=\lambda^i \otZ_{y,i}(K)=\lambda^i\oZ_i(\ell_{y}+\ind_K).
$$
As a conclusion, we have the following statement: there exist continuous and epi-translation invariant valuations $\oZ_0,\dots,\oZ_n$ on $\fconvs$ such that, for every $y\in\R^n$ and for every $K\in\cK^n$, setting $u=\ell_{y}+\ind_K$, we have
$$
\oZ(u)=\sum_{i=0}^n\oZ_i(u),
$$
and, for every $\lambda\geq 0$,
$$
\oZ_i(\lambda \sq u)=\lambda^i\oZ_i(u).
$$
The same statement holds if we replace $u=\ell_{y}+\ind_K$ by $u=\ell_{y}+\ind_K+\alpha$, for any constant $\alpha\in\R$
as all valuations involved are vertically translation invariant.

If we apply Lemma \ref{null lemma} to 
$$
\oZ-\sum_{i=0}^n\oZ_i,
$$
we get that this valuation vanishes on $\fconvs$, so that
$$
\oZ(u)=\sum_{i=0}^n\oZ_i(u)
$$
for every $u\in\fconvs$.
For $\lambda\geq 0$, the same lemma applied to the valuation on $\fconvs$ defined by
$$
u\mapsto \oZ_i(\lambda \sq u)-\lambda^i \oZ_i(u),
$$
shows that this must be identically zero as well, that is, $\oZ_i$ is epi-homogeneous of degree $i$. The proof is complete.

\section{Polynomiality}\label{section polynomiality}

In this section we establish the polynomial behavior of continuous and epi-translation invariant valuations on $\fconvs$. This corresponds to the polynomiality of translation invariant valuations on convex bodies stated by Hadwiger and proved by McMullen \cite{McMullen77}. We start by recalling the definition of inf-convolution (see, for example,  \cite{RockafellarWets,Schneider:CB2}). For
$u,v\in\fconvx$, we define the function $u\infconv v\colon\R^n\to [-\infty,+\infty]$ by
$$
u\infconv v(z)=\inf\{u(x)+v(y)\colon x,y\in\R^n,\, x+y=z\}
$$
for $z\in\R^n$. This operation can be extended to more than two functions with corresponding coefficients. 
The inf-convolution has a straightforward geometric meaning: the epigraph of $u\infconv v$ is the Minkowski sum of the epigraphs
of $u$ and $v$.  

By \cite[Section 1.6]{Schneider:CB2}, for every $\alpha,\beta>0$ and for every $u,v\in\fconvx$, we have\linebreak$\alpha\sq u\infconv\beta\sq u\in\fconvx$, if this function does not attain $-\infty$. Moreover, in this case we have the following relation (see for instance \cite[Proposition 2.1]{ColesantiFragala}):
\begin{equation*}
(\alpha\sq u\infconv\beta\sq v)^*=(\alpha u^*+\beta v^*).
\end{equation*}
This shows in particular that if $u,v\in\fconvs$ then $\alpha\sq u\infconv\beta\sq v\in\fconvs$. Indeed, in this case $u^*$ and $v^*$ belong to $\fconvf$
and so does their usual sum. Consequently, its conjugate belongs to $\fconvs$. We say that $\oZ$ is \emph{epi-additive} if 
$$\oZ(\alpha\sq u\infconv\beta\sq v)=\alpha \oZ(u)+\beta\oZ(v)$$
for all $\alpha,\beta>0$ and $u,v\in\fconvs$.

Let $\oZ\colon\fconvs\to\R$ be a continuous, epi-translation invariant valuation that is 
epi-homogeneous of degree $m\in\{1,\ldots,n\}$. For $u_1\in\fconvs$, we consider the functional $\oZ_{u_1}\colon\fconvs\to\R$ defined by
$$
\oZ_{u_1}(u)=\oZ(u \infconv u_1).
$$
The functional $\oZ_{u_1}$ is a continuous and epi-translation invariant valuation on $\fconvs$. Indeed, the valuation property,  continuity and vertical translation invariance follow immediately from the corresponding properties of $\oZ$. As for translation invariance, let $x_0\in\R^n$ and $\tau:\R^n\to\R^n$ be the translation by $x_0$, that is, $\tau(x)=x+x_0$. We have
\begin{eqnarray*}
(u\circ \tau^{-1}) \infconv u_1 &=&\left((u\circ \tau^{-1})^* +u_1^* \right)^*=
\left(u^*+\langle \cdot,x_0\rangle + u_1^*\right)^*=(u\infconv u_1)\circ\tau^{-1}.
\end{eqnarray*}
Hence  the epi-translation invariance of $\oZ_{u_1}$ follows from the epi-translation invariance of $\oZ$. Therefore, we may apply Theorem~\ref{McM sc} to obtain a polynomial expansion
$$\oZ((\lambda \sq u) \infconv  u_1) = \oZ_{u_1}(\lambda \sq u ) = \sum_{i=0}^n \lambda^i \oZ_{u_1,i}(u)$$
for $\lambda\geq 0$ and $u\in\fconvs$, where the functionals $\oZ_{u_1,i}$ are continuous, epi-translation invariant valuations on $\fconvs$ that are epi-homogeneous of degree $i\in\{0,\ldots,n\}$.

Similarly, for fixed $\bar{u}\in\fconvs$ one can show that $v\mapsto \oZ_{v,i}(\bar{u})$ defines a continuous and epi-translation invariant valuation on $\fconvs$. Hence, as in  the proof of Theorem 6.3.4 in \cite{Schneider:CB2},  we may repeat this argument to obtain the following statement.

\goodbreak
\begin{theorem}\label{thm poly sc} Let $\,\oZ\colon\fconvs\to\R$ be a continuous and epi-translation invariant valuation  that is 
epi-homogeneous of degree $m\in\{1,\ldots,n\}$. There exists a symmetric function $\obZ:(\fconvs)^m\to\R$ such that for $k\in\N$, $u_1,\ldots,u_k\in\fconvs$ and $\lambda_1,\ldots,\lambda_k\geq 0$,
$$
\oZ(\lambda_1\sq u_1\infconv\cdots \infconv\lambda_k\sq u_k) = \sum_{\substack{i_1,\ldots, i_k\in \{0,\ldots,m\}\\i_1+\cdots+ i_k = m}} \binom{m}{i_1 \cdots i_k} \lambda_1^{i_1} \cdots \lambda_k^{i_k} \obZ(u_1 [i_1],\ldots,u_k [i_k]),
$$
where $u_j[i_j]$ means that the argument $u_j$ is repeated $i_j$ times.
Moreover, the function $\obZ$ is epi-additive in each variable. For $i\in\{1,\ldots,m\}$ and $u_{i+1}, \dots, u_m\in\fconvs$, the map $u\mapsto\obZ (u[i],u_{i+1},\ldots,u_m)$ is a continuous, epi-translation invariant valuation on $\fconvs$ that is epi-homogeneous of degree $i$.
\end{theorem}

The special case $m=1$ in the previous result leads to the following result.

\begin{corollary}\label{add} 
If $\,\oZ\colon\fconvs\to\R$ is a continuous and epi-translation invariant valuation that is epi-homogeneous of degree 1, then  $\oZ$ is epi-additive.
\end{corollary}

\goodbreak
Finally, we also obtain the dual statements. We say that a functional $\oZ: \fconvf \to \R$ is \emph{additive} if $\oZ(\alpha\,v+\beta\,w)= \alpha \oZ(v)+ \beta\oZ(w)$ for all $\alpha, \beta \ge 0$ and $v,w\in\fconvf$.

\begin{theorem}\label{thm poly f} Let $\,\oZ\colon\fconvf\to\R$ be a continuous, dually epi-translation invariant valuation  that is 
homogeneous of degree $m\in\{1,\ldots,n\}$. There exists a symmetric function $\obZ:(\fconvf)^m\to\R$ such that for $k\in\N$, $v_1,\ldots,v_k\in\fconvf$ and $\lambda_1,\ldots,\lambda_k\geq 0$,
$$
\oZ(\lambda_1\, v_1+\cdots +\lambda_k\, v_k) = \sum_{\substack{i_1,\ldots, i_k\in \{0,\ldots,m\}\\i_1+\cdots+ i_k = m}} \binom{m}{i_1 \cdots i_k} \lambda_1^{i_1} \cdots \lambda_k^{i_k} \obZ(v_1 [i_1],\ldots,v_k [i_k]).
$$
Moreover, the function $\obZ$ is additive in each variable. For $i\in\{1,\ldots,m\}$ and $v_{i+1}, \dots, v_m\in\fconvf$,  the map $v\mapsto\obZ (v[i],v_{i+1},\ldots,v_m)$ is a continuous and dually epi-translation invariant valuation on $\fconvf$ that is homogeneous of degree $i$.
\end{theorem}

\goodbreak
The special case $m=1$ in the previous result leads to the following result.

\begin{corollary}
If $\,\oZ\colon\fconvf\to\R$ is a continuous and dually epi-translation invariant valuation that is homogeneous of degree 1, then  $\oZ$ is additive.
\end{corollary}

Let $\zeta\in C_c(\R^n)$. By Proposition~\ref{if part - general case}, the functional
$$\oZ(v)=\int_{\R^n\times \R^n} \zeta(x) \d \Theta_0(v,(x,y))$$
defines a continuous, dually epi-translation invariant valuation on $\fconvf$ that is homogeneous of degree $n$. Hence, by Theorem~\ref{thm poly f}, for  $v_1,\ldots,v_k\in\fconvf$ and $\lambda_1,\ldots,\lambda_k\geq 0$, there exists a symmetric function $\bar{\oZ}:(\fconvf)^n\to\R$ such that
$$
\oZ(\lambda_1\, v_1+\cdots +\lambda_k\, v_k) = \sum_{\substack{i_1,\ldots, i_k\in \{0,\ldots,n\}\\i_1+\cdots+ i_k = n}} \binom{n}{i_1 \cdots i_k} \lambda_1^{i_1} \cdots \lambda_k^{i_k} \obZ(v_1 [i_1],\ldots,v_k [i_k]).
$$
If we assume in addition that $v_1,\ldots,v_k\in C^2(\R^n)$, then by (\ref{smooth}) and properties of the mixed discriminant, we can also write
\begin{align*}
\oZ(\lambda_1\, v_1+\cdots +\lambda_k\, v_k) 
&= \int_{\R^n} \zeta(x)\, \det (\Hess\,(\lambda_1\, v_1+\cdots +\lambda_k\, v_k)(x)) \d x\\
&= \sum_{i_1,\ldots,i_n=1}^k \lambda_{i_1}\cdots \lambda_{i_n} \int_{\R^n} \zeta(x)\,\det (\Hess v_{i_1}(x),\ldots,\Hess v_{i_n}(x)) \d x.
\end{align*}
It is now easy to see that for such functions $v_1,\ldots,v_k$ and $i_1,\ldots, i_k\in\{0,\ldots,n\}$ with $ i_1+\cdots+i_k=n$,
$$\bar{\oZ}(v_1[i_1],\ldots,v_k[i_k]) = \int_{\R^n} \zeta(x)\,\det (\Hess v_1(x)[i_1],\ldots,\Hess v_k [i_k]) \d x.$$
Note that this is a special case of \eqref{val A}.

\section{Classification Theorems}\label{class}

The classification of valuations that are epi-homogenous of degree 0 is straightforward.

\begin{theorem}\label{theorem 0-epi-homogeneous} 
A functional $\oZ\colon\fconvs\to\R$ is a continuous and epi-translation invariant valuation  that is epi-homogeneous of degree $0$, if and only if  $\,\oZ$ is constant.
\end{theorem}

\begin{proof}
Let $\oZ\colon\fconvs\to\R$ be a continuous and epi-translation invariant valuation that is epi-homogeneous of degree zero. We show that $\oZ$ is constant.
Indeed, for given $y\in\R^n$, the functional $\otZ_y\colon\cK^n\to\R$ defined by
$$
\otZ_y(K)=\oZ(\ell_{y}+\ind_K).
$$
is a zero-homogeneous, continuous and translation invariant valuation on $\cK^n$ and therefore constant. Such a constant cannot depend on $y$, as, choosing $K=\{0\}$, we obtain
$$
\ind_{\{0\}}+\ell_{y}=\ind_{\{0\}}+\ell_{y_0}
$$
for all $y,y_0\in\R^n$. Hence there exists $\alpha\in\R$ such that
$$
\oZ(\ind_K+\ell_{y})=\alpha
$$
for all $K\in\cK^n$ and $y\in\R^n$. Thus the statement follows from applying Lemma \ref{null lemma} to $\oZ-\alpha$.
\end{proof}

\goodbreak
By duality, we also obtain the following result.

\begin{theorem}\label{theorem 0-homogeneous} 
A functional $\oZ\colon\fconvf\to\R$ is a continuous and dually epi-translation invariant valuation  that is homogeneous of degree $0$, if and only if  $\,\oZ$ is constant.
\end{theorem}

Next, we prove Theorem \ref{theorem n-homogeneous}. 
The ``if'' part of the proof follows from Proposition \ref{prop 1} and the subsequent remark. The proof of the theorem is completed by the next statement.

\begin{proposition}\label{only if part - volume} If $\,\oZ: \fconvs \to \R$ is a continuous, epi-translation invariant valuation, that is epi-homogeneous of degree $n$, then there exists $\zeta\in C_c(\R^n)$ such that
$$
\oZ(u)=\int_{\dom(u)}\zeta(\nabla u(x))\d x$$
for every $u\in\fconvs$.
\end{proposition}

\begin{proof} For $y\in\R^n$, we consider the map $\otZ_y\colon\cK^n\to\R$ defined by
$$
\otZ_y(K)=\oZ(\ell_{y}+\ind_K).
$$
We know from the proof of Theorem \ref{McM sc} that $\otZ_y$ is a continuous and translation invariant valuation on $\cK^n$. Moreover, as the functional $\oZ$ is 
epi-homogeneous of degree $n$, the functional $\otZ_y$ is homogeneous of degree $n$.
By Theorem \ref{n-homogeneous}, for each $y\in\R^n$, there exists a constant,
 that we denote by $\zeta(y)$, such that
\begin{equation}\label{representation}
\otZ(K)=\zeta(y)\,V_n(K)
\end{equation}
for every $K\in\cK^n$. As $\oZ$ is continuous, the function $\zeta: \R^n\to \R$ is continuous.
We prove, by contra\-diction, that $\zeta$ has compact support. Assume that there exists a sequence $y_k\in\R^n$,
such that 
\begin{equation}\label{norm to infinity}
\lim_{k\to\infty}\vert  y_k\vert =+\infty
\end{equation}
and $\zeta(y_k)\ne0$ for every $k$. Without loss of generality, we may assume that 
\begin{equation}\label{to en}
\lim_{k\to\infty}\frac{y_k}{\vert  y_k\vert }=e_n
\end{equation}
where $e_n$ is the $n$-th element of the canonical basis of $\R^n$.

Let 
$$
B_k=\{x\in y_k^\perp\colon\vert  x\vert \le1\}, \quad B_\infty=\{x\in e_n^\perp\colon\vert  x\vert \le1\}.
$$
Define the cylinder
$$
C_k=\left\{x+ty_k\colon x\in B_k, t\in\left[0,\frac1{\zeta(y_k)}\right]\right\}.
$$
We have
$$
V_n(C_k)=\frac{\kappa_{n-1}}{\zeta(y_k)},
$$
where $\kappa_{n-1}$ is the $(n-1)$-dimensional volume of the unit ball in $\R^{n-1}$. 

For $k\in\N$, we consider the function
$$
u_k=\ell_{y_k}+\ind_{C_k}.
$$
This is a sequence of functions in $\fconvs$; using \eqref{norm to infinity} and \eqref{to en}, it follows from Lemma~\ref{hd conv lvl sets} that $u_k$ epi-converges to 
$$
u_\infty=\ind_{B_\infty}.
$$
In particular, by the continuity of $\oZ$ and \eqref{representation} we get
$$
0=\oZ(u_\infty)=\lim_{k\to\infty}\oZ(u_k).
$$
On the other hand, by the definition of $u_k$ and \eqref{representation},
$$
\oZ(u_k)=\zeta(y_k)\, V_n(C_k)=\kappa_{n-1}>0.
$$
This completes the proof.\end{proof}

\section{Valuations without Vertical Translation Invariance}\label{vert}

In this part we see that Theorems \ref{McM sc} and \ref{McM finite} are no longer true if we remove the assumption of vertical translation invariance. To do so, 
on the base of Theorem \ref{theorem 1} we construct the following example. For  $\eta \in C_c(\R^n)$ and $v\in\fconvf\cap C^2(\R^n)$, define 
\begin{equation}\label{val nh}
\oZ(v)=\int_{\R^n}e^{v(x)-\langle\nabla v(x),x\rangle}\,\eta(x)\det(\Hess v(x)) \d x.
\end{equation}
By Theorem \ref{theorem 1}, the functional defined in \eqref{val nh} can be extended to a continuous valuation on $\fconvf$. It is dually translation invariant but not vertically translation invariant. We choose $v\in\fconvf$ as
$$
v(x)=\tfrac12{\vert  x\vert ^2}.
$$
Note that the Hessian matrix of $v$ is everywhere equal to the identity matrix. Hence $\det(\Hess v)=1$  on $\R^n$. For $\lambda\geq 0$ we have
$$
\oZ(\lambda v)=\lambda^n\int_{\R^n} \eta(x)e^{\lambda \frac{\vert  x\vert ^2}{2}} \d x.
$$
If $\eta$ is non-negative and $\eta(x)\ge 1$ for every $x$ such that $1\le\vert  x\vert \le2$, then
$$
\oZ(\lambda v)\ge c\lambda^n e^{\lambda/2}
$$
for a suitable constant $c>0$ and for every $\lambda\geq 0$. Hence $\oZ(\lambda v)$ does not have  polynomial growth as $\lambda$ tends to $\infty$.

\begin{theorem} There exist continuous, dually translation invariant valuations on $\fconvf$ which cannot be written as  finite sums of
homogeneous valuations. 
\end{theorem}

\goodbreak
As a consequence we also have the following dual statement.

\begin{theorem} There exist  continuous, translation invariant valuations on $\fconvs$ which cannot be written as finite sums of
epi-homogeneous valuations. 
\end{theorem}

\goodbreak
\section{Epi-translation Invariant Valuations on Coercive Functions}\label{coercive}

In this part we prove that every continuous and epi-translation invariant valuation on $\fconv$ is trivial. 

\begin{theorem}\label{coer}
Every continuous, epi-translation invariant valuation $\oZ: \fconv\to\R$ is constant. 
\end{theorem}
\begin{proof}
Let $\oZ:\fconv\to\R$ be a continuous, epi-translation invariant valuation. We need to show that there exists $\alpha\in\R$ such that $\oZ(u)=\alpha$ for every $u\in\fconv$. As in the proof of Theorem~\ref{McM sc} define for $y\in\R^n\backslash\{0\}$ the map ${\otZ}_y:\cK^n\to\R$ by
$${\otZ}_y(K)=\oZ(\ell_{y}+\ind_K)$$
for every $K\in\cK^n$. Since ${\otZ}_y$ is a continuous and translation invariant valuation, by Theorem \ref{McM} it admits a homogeneous decomposition
$${\otZ}_y = \sum_{j=0}^n \tilde{\oZ}_{y,j},$$
where each  $\tilde{\oZ}_{y, j}$ is a continuous, translation invariant valuation on $\cK^n$ that is homogeneous of degree $j$.\par
Next, we will show that $\tilde{\oZ}_{y,j}\equiv 0$ for all $1\leq j \leq n$. Since
$$\tilde{\oZ}_{y,0}(K)=\lim_{\lambda \to 0} \tilde{\oZ}_{y,0}(\lambda K) = \tilde{\oZ}_{y,0} (\{0\})$$
for every $K\in\cK^n$, this will then imply that $\tilde{\oZ}_y$ is constant. By continuity it is enough to show that $\tilde{\oZ}_{y,j}$ vanishes on polytopes for all $1\leq j \leq n$. Since $\tilde{\oZ}_y$ is continuous, it is enough to restrict to polytopes with no facet parallel to $y^\bot$. Therefore, fix such a polytope $P\in\cP^n$ of dimension at least one. By translation invariance we can assume that the origin is one of the vertices of $P$ and that $P$ lies in the half-space $\{x\in\R^n\,:\,\langle x,y\rangle\geq 0\}$. In particular, this gives $P\cap y^\bot = \{0\}$, $\langle x,y\rangle>0$ for all $x\in P\backslash \{0\}$ and moreover $\langle x,y\rangle>0$ and for all $x\in \lambda P \backslash \{0\}$ for all $\lambda >0$. Due to the choice of $P$ we obtain that $\ell_y+\ind_{\lambda P}$ is epi-convergent to $\ell_y+\ind_{C}$ as $\lambda \to \infty$ where $C$ is the infinite cone over $P$ with apex at the origin, that is $C$ is the positive hull of $P$. Furthermore $\ell_y+\ind_{C}\in\fconv$ since $y\neq 0$. By continuity this gives
$$\oZ(\ell_y+\ind_{C}) = \lim_{\lambda \to \infty} \oZ(\ell_y+\ind_{\lambda P}) = \lim_{\lambda \to \infty} \tilde{\oZ}_y(\lambda P) = \lim_{\lambda \to \infty} \sum_{j=0}^n \lambda^j \tilde{\oZ}_{y,j}(P).$$
Since the left side of this equation is finite, we have  $\tilde{\oZ}_{y,n}(P) = 0$. Otherwise, the right side would be $\pm \infty$, depending on the sign of $\tilde{\oZ}_{y,n}(P)$. Since $P$ was arbitrary, we obtain that $\tilde{\oZ}_{y,n}$ vanishes on all compact convex polytopes of dimension greater or equal
than $1$ and by continuity $\tilde{\oZ}_{y,n}\equiv 0$. Similarly, one can now show by
induction that also $\tilde{\oZ}_{y,j}\equiv 0$ for all $1\leq j \leq n-1$.\par
\goodbreak
We have proven so far that for every $y\in\R^n\backslash\{0\}$ there exists a constant $\alpha(y) \in\R$ such that $\tilde{\oZ}_y\equiv \alpha(y)$. Since
$$
\alpha(y)= \tilde{\oZ}_y(\{0\}) = \oZ(\ind_{\{0\}}),
$$
we obtain that $\alpha(y)$ is in fact independent of $y$, that is, there exists $\alpha\in\R$ such that $\tilde{\oZ}_y\equiv \alpha$ for every $y\in\R^n$. 
By the definition of $\tilde{\oZ}_y$ and the vertical translation invariance of $\tilde{\oZ}$ this gives $\oZ(\ell_y + \ind_K +\beta)=\alpha$ for every $K\in\cK^n$, $y\in\R^n\backslash\{0\}$ and $\beta\in\R$. The claim now follows from Lemma~\ref{null lemma}.
\end{proof}

If $u\in\fconv$, then its conjugate $u^*\in\fconvx$ and the origin is an interior point of its domain (see, for example, \cite[Theorem 11.8]{RockafellarWets}). Let $\fconvi$ be the set of functions in $\fconvx$ with the origin in the interior of its domain. Theorem \ref{coer} has the following dual.

\begin{theorem}
Every continuous, dually epi-translation invariant valuation $\oZ: \fconvi\to\R$ is constant. 
\end{theorem}

\bigskip


\end{document}